\documentclass[]{interact}

\usepackage{epstopdf}% To incorporate .eps illustrations using PDFLaTeX, etc.
\usepackage[caption=false]{subfig}% Support for small, `sub' figures and tables

\usepackage[numbers,sort&compress]{natbib}% Citation support using natbib.sty
\bibpunct[, ]{[}{]}{,}{n}{,}{,}% Citation support using natbib.sty
\makeatletter% @ becomes a letter
\def\NAT@def@citea{\def\@citea{\NAT@separator}}% Suppress spaces between citations using natbib.sty
\makeatother% @ becomes a symbol again

\theoremstyle{plain}% Theorem-like structures provided by amsthm.sty
\newtheorem{theorem}{Theorem}[section]
\newtheorem{lemma}[theorem]{Lemma}

\newtheorem{problem}[theorem]{Problem}

\theoremstyle{definition}
\newtheorem{definition}[theorem]{Definition}
\newtheorem{example}[theorem]{Example}

\theoremstyle{remark}
\newtheorem{remark}{Remark}

\newcommand{\beq}{\begin{equation}} \newcommand{\eeq}{\end{equation}}
\newcommand{\bea}{\begin{eqnarray}} \newcommand{\eea}{\end{eqnarray}}
\newcommand{\bear}{\begin{eqnarray*}} \newcommand{\eear}{\end{eqnarray*}}

\usepackage{hyperref}
\hypersetup{colorlinks=true, linkcolor=blue, citecolor=blue, filecolor=blue, urlcolor=blue,
            pdfproducer={Latex},
            pdfcreator={pdflatex}}

\begin{document}

\title{Constants of motion for  Isoperimetric Variational Problems with Time Delay}

\author{
\name{G. S. F. Frederico\textsuperscript{a}, M.~J.~Lazo\textsuperscript{b}\thanks{CONTACT M.~J.~Lazo. Email: matheuslazo@furg.br}, M. N. F. Barreto\textsuperscript{a} and J. Paiva\textsuperscript{b}}
\affil{\textsuperscript{a}Federal University of Cear\'a, Campus de Russas, Russas, Brazil; \textsuperscript{b}Federal University of Rio Grande, Rio Grande - RS, Brazil.}
}

\maketitle

%\ead{gastao.frederico@ua.pt}

% ------------------------

\begin{abstract}
In the present work we obtain the constants of motion for isoperimetric variational problems with time delay. We consider a constrained optimization problem where the Lagrangian function defining the functional depends on time delayed arguments. We prove the isoperimetric Euler--Lagrange and DuBois--Reymond type optimality conditions and, in order to investigate the constants of motion for this problem, we obtain a nonsmooth extension of Noether's symmetry theorem for isoperimetric variational problems with delayed arguments.

\end{abstract}

%\begin{keyword}
%isoperimetric problems \sep time delays \sep symmetries \sep conservation laws \sep Euler--Lagrange and DuBois--Reymond necessary optimality condition \sep Noether's theorem.

%{2010 Mathematics Subject Classification:} 49K05; 49S05.

%\end{keyword}

\maketitle

% ------------------------

\section{Introduction}

One of the oldest and interesting classes of variational problems, with applications in several fields, is the isoperimetric problems \cite{Bruce:book}. The isoperimetric problem in mathematical physics has roots in the Queen Dido problem of the calculus of variations and has been subject to several investigations for a long time. The original Dido's problem was posed by the ancient Greeks: find the closed plane curve of a given length that encloses the largest area. The reason this problem is called isoperimetric is that one is maximizing the area inside the curve while keeping the perimeter fixed. More generally, an isoperimetric variational problem is a constrained optimization problem where one is trying to extremise a functional subjected to a functional constraint. 

On the other hand, the concept of symmetry plays an important role in science and engineering. Symmetries are described by transformations, which result in the same object after the transformations are carried out. They are described mathematically by parameter groups of transformations \cite{Logan:b,CD:Gel:1963}. In the calculus of variation their importance, as recognized by Noether in 1918 \cite{Noether:1971}, is connected with the existence of conservation laws that can be used to reduce the order of the Euler--Lagrange differential equations \cite{MR2351636}. Noether's symmetry theorem is nowadays recognized as one of the most beautiful results of the calculus of variations and optimal control \cite{CD:Djukic:1972,CD:Kosmann:2004}, and it becomes one of the most important theorems for physics in the 20th century. Since the seminal work of Emmy Noether, it is well known that all conservations laws in mechanics, \textrm{e.g.}, conservation of energy or conservation of momentum, are directly related to the invariance of the action functional under a family of transformations.

Within the years, the Noether's theorem has been studied by many authors and generalized in different directions: see \cite{CD:Bartos,Gastao:PhD:thesis,GastaoLazo,GF:JMAA:07,MR2323264,book:frac,NataliaNoether,ejc} and references therein. In particular, in a recent paper \cite{GF2012}, a Noether's theorem was formulated for variational problems with delayed arguments. The result is important because problems with delays play a crucial role in the modeling of real-life phenomena in various fields of applications \cite{ChiLoi,EFridman,GoKeMa}. In order to prove Noether's theorem with delays, it was assumed that admissible functions are $\mathcal{C}^2$-smooth and that Noether's conserved quantity holds along all $\mathcal{C}^2$-extremals of the Euler--Lagrange equations with time delay \cite{GF2012}.

In the present work, we obtained the time delay isoperimetric Euler--Lagrange and DuBois--Reymond type optimality conditions, we prove the existence and regularity of the minimizer for the problem, and finally we also obtain a nonsmooth extension of Noether's symmetry theorem for time delay isoperimetric variational problems. The Euler-Lagrange and DuBois-Reymond type equations we obtained generalizes the previous result found for variational problems with delayed arguments \cite{GF2012}. Furthermore, we extend the Noether's theorem with delays for time delay isoperimetric variational problems. In order to illustrate some applications, we present simple examples.

The text is organized as follows. In Section~\ref{sec:prelim} the fundamental problem of variational calculus with delayed arguments is formulated and a short review of the results for $\mathcal{C}^2$-smooth admissible functions is given. The main contributions of the paper appear in Sections~\ref{Main results}: we prove the isoperimetric Euler--Lagrange and DuBois--Reymond type optimality conditions (Theorem~\ref{Thm:FractELeq1} and
Theorem~\ref{theo:cdrnd}, respectively), and an isoperimetric Noether's symmetry theorem with time delay (Theorem~\ref{theo:tnnd}). Finally, the conclusions and open questions are presented in Section~\ref{sec:Conc}.

\section{Preliminaries}
\label{sec:prelim}

In this section, we review some necessary results on the calculus of variations with time delay. For more on variational problems with delayed arguments we refer the reader to
\cite{Basin:book,Bok,GoKeMa,DH:1968,Kra,Kharatishvili,Ros}.

The fundamental problem consists of minimizing a functional
\begin{equation}
\label{Pe}
J^{\tau}[q(\cdot)] = \int_{t_{1}}^{t_{2}}
L\left(t,q(t),\dot{q}(t),q(t-\tau),\dot{q}(t-\tau)\right) dt
\end{equation}
subject to boundary conditions
\begin{equation}
\label{Pe2}
q(t)=\delta(t)~\textnormal{ for }~t\in[t_{1}-\tau,t_{1}]~\textnormal{ and }~q(t_2)=q_{t_2}.
\end{equation}
We assume that the Lagrangian
$L :[t_1,t_2] \times (\mathbb{R}^{n})^4 \rightarrow \mathbb{R}$, $n\in\mathbb{N}$,
%$L :[t_1,t_2] \times \mathbb{R}^{4} \rightarrow \mathbb{R}$,
is a $\mathcal{C}^{2}$-function with respect to all its arguments,
the admissible functions $q(\cdot)$ are $\mathcal{C}^2$-smooth, $t_{1}< t_{2}$ are fixed in $\mathbb{R}$, $\tau$ is a given positive real number such that $\tau<t_{2}-t_{1}$, and $\delta$ is a given piecewise smooth function on $[t_1-\tau,t_1]$. Throughout the text, $\partial_{i}L$ denotes the partial derivative of $L$ with respect to its $i$th argument, $i=1,\dots,5$ and $\dot{q}=\frac{dq}{dt}\,.$ For convenience of notation, we introduce the operator $[\cdot]_{\tau}$ defined by
$$
[q]_{\tau}(t)=(t,q(t),\dot{q}(t),q(t-\tau),\dot{q}(t-\tau)).
$$

The next theorem gives a necessary optimality condition of Euler--Lagrange type
for \eqref{Pe}--\eqref{Pe2}.

\begin{theorem}[Euler--Lagrange equations with time delay \cite{DH:1968}]
\label{th:EL1}
If $q(\cdot)\in\mathcal{C}^2$ is a minimizer for problem \eqref{Pe}--\eqref{Pe2},
then $q(\cdot)$ satisfies the following Euler--Lagrange equations with time delay:
\begin{equation}
\label{EL1}
\begin{cases}
\frac{d}{dt}\left\{\partial_{3}L[q]_{\tau}(t)+
\partial_{5}L[q]_{\tau}(t+\tau)\right\}
=\partial_{2}L[q]_{\tau}(t)+\partial_{4}L[q]_{\tau}(t+\tau),
\quad t_{1}\leq t\leq t_{2}-\tau,\\
\frac{d}{dt}\partial_{3}L[q]_{\tau}(t) =\partial_{2}L[q]_{\tau}(t),
\quad t_{2}-\tau< t\leq t_{2}.
\end{cases}
\end{equation}
\end{theorem}

\begin{definition}[Weak Extremals]
\label{def:scale:ext}
The solutions $q(\cdot)$ of the Euler--Lagrange equations \eqref{EL1} with time delay are called \emph{weak extremals}.
\end{definition}

\begin{definition}[Invariance of \eqref{Pe}]
\label{def:invnd}
Consider the following $s$-parameter group of infinitesimal transformations:
\begin{equation}
\label{eq:tinf}
\begin{cases}
\bar{t} = t + s\eta(t,q) + o(s) \, ,\\
\bar{q}(t) = q(t) + s\xi(t,q) + o(s),
\end{cases}
\end{equation}
where
$\eta\in \mathcal{C}^1(\mathbb{R}^{n+1},\mathbb{R})$
and $\xi\in \mathcal{C}^1(\mathbb{R}^{n+1},\mathbb{R}^n)$.
%$\eta\in \mathcal{C}^1(\mathbb{R}^2)$
%and $\xi\in \mathcal{C}^1(\mathbb{R}^2)$.
Functional \eqref{Pe} is said to be invariant under \eqref{eq:tinf} if
\begin{multline*}
\label{eq:invnd}
0 = \frac{d}{ds}
\int_{\bar{t}(I)} L\left(t+s\eta(t,q(t))+o(s),q(t)+s\xi(t,q(t))+o(s),
\frac{\dot{q}(t)+s\dot{\xi}(t,q(t))}{1+s\dot{\eta}(t,q(t))},\right.\\
\left.q(t-\tau)+s\xi(t-\tau,q(t-\tau))+o(s),\frac{\dot{q}(t-\tau)
+s\dot{\xi}(t-\tau,q(t-\tau))}{1+s\dot{\eta}(t-\tau,q(t-\tau))}\right)
(1+s\dot{\eta}(t,q(t))) dt \Biggl.\Biggr|_{s=0}
\end{multline*}
for any  subinterval $I \subseteq [t_1,t_2]$.
\end{definition}

\begin{definition}[Constant of motion/conservation law with time delay]
\label{def:leicond}
We say that a quantity
$C(t,t+\tau,q(t),q(t-\tau),q(t+\tau),\dot{q}(t),\dot{q}(t-\tau),\dot{q}(t+\tau))$ is
a constant of motion with time delay $\tau$ if
\begin{equation}
\label{eq:conslaw:td1}
\frac{d}{dt} C(t,t+\tau,q(t),q(t-\tau),q(t+\tau),\dot{q}(t),\dot{q}(t-\tau),\dot{q}(t+\tau))= 0
\end{equation}
along all the weak extremals $q(\cdot)$ (\textrm{cf.} Definition~\ref{def:scale:ext}).
The equality \eqref{eq:conslaw:td1} is then a conservation law with time delay.
\end{definition}

The next theorem extends the DuBois--Reymond necessary optimality condition to problems of the calculus of variations with time delay.

\begin{theorem}[DuBois--Reymond necessary conditions with time delay \cite{GF2012}]
\label{theo:cdrnd0} If $q(\cdot)\in\mathcal{C}^2$ is a weak extremal of the
functional \eqref{Pe} subject to \eqref{Pe2} such that
\begin{equation*}
\partial_4L[q]_{\tau}(t+\tau)\cdot \dot{q}(t)+\partial_5L[q]_{\tau}(t+\tau)\cdot
\ddot{q}(t)=0
\end{equation*}
for all $t\in[t_1-\tau,t_2-\tau]$, then it satisfies the following
conditions:
\begin{equation}
\label{eq:cdrnd0}
\begin{cases}
\frac{d}{dt}\left\{L[q]_{\tau}(t)-\dot{q}(t)\cdot(\partial_{3} L[q]_{\tau}(t)
+\partial_{5} L[q]_{\tau}(t+\tau))\right\} = \partial_{1} L[q]_{\tau}(t),\quad
t_1\leq t\leq t_{2}-\tau,\\
\frac{d}{dt}\left\{L[q]_{\tau}(t)
-\dot{q}(t)\cdot\partial_{3} L[q]_{\tau}(t)\right\}
=\partial_{1} L[q]_{\tau}(t),\quad t_2-\tau\leq t\leq t_{2}\,.
\end{cases}
\end{equation}
\end{theorem}

\begin{remark}
If we assume that admissible functions in problem \eqref{Pe}--\eqref{Pe2} are Lipschitz continuous, then one can show that the DuBois--Reymond necessary conditions with time delay \eqref{eq:cdrnd} are still valid (cf. \cite{GF2013}).
\end{remark}

Theorem~\ref{theo:tnnd1} establishes an extension of Noether's
theorem to problems of the calculus of variations with time delay.

\begin{theorem}[Noether's symmetry theorem with time delay \cite{GF2012}]
\label{theo:tnnd1} If functional \eqref{Pe} is invariant in the
sense of Definition~\ref{def:invnd} such that
\begin{equation*}
\partial_4L[q]_{\tau}(t+\tau)\cdot \dot{q}(t)+\partial_5L[q]_{\tau}(t+\tau)\cdot
\ddot{q}(t)=0
\end{equation*}
and, then the quantity
$C(t,t+\tau,q(t),q(t-\tau),q(t+\tau),\dot{q}(t),\dot{q}(t-\tau),\dot{q}(t+\tau))$
defined by
\begin{multline}
\label{eq:tnnd1}
\left(\partial_{3} L[q]_{\tau}(t)
+\partial_{5} L[q]_{\tau}(t+\tau)\right)\cdot\xi(t,q(t))\\
+\Bigl(L[q]_{\tau}(t)-\dot{q}(t)\cdot(\partial_{3} L[q]_{\tau}(t)
+\partial_{5} L[q]_{\tau}(t+\tau))\Bigr)\eta(t,q(t))
\end{multline}
for $t_1\leq t\leq t_{2}-\tau$ and by
\begin{equation}
\label{eq:tnnd2}
\partial_{3} L[q]_{\tau}(t)\cdot\xi(t,q(t))
+\Bigl(L[q]_{\tau}(t)-\dot{q}(t)\cdot\partial_{3} L[q]_{\tau}(t)\Bigr)\eta(t,q(t))
\end{equation}
for $t_2-\tau< t\leq t_{2}\,,$ is a constant of motion with time delay
(\textrm{cf.} Definition~\ref{def:leicond}).
\end{theorem}

% ------------------------
\section{Main results}
\label{Main results}

Along the work we have $1 < p < \infty$, and $p'$ denotes the adjoint of $p\,$.
Let $\left\|\cdot\right\|$ be the standard Euclidean norm of $\mathbb{R}^n$.
For any $1 \leq r \leq \infty$, we denote
\begin{itemize}
\item by $\mathrm{L}^r := \mathrm{L}^r (t_1,t_2;\mathbb{R}^n)$ the usual space of $r$-Lebesgue integrable functions endowed with its usual norm $\Vert \cdot \Vert_{\mathrm{L}^r}$;
\item by $\mathrm{W}^{1,r} := \mathrm{W}^{1,r} (t_1,t_2;\mathbb{R}^n)$ the usual $r$-Sobolev space endowed with its usual norm $\Vert \cdot \Vert_{\mathrm{W}^{1,r}}$;
\item by $\mathcal{C}:=  \mathcal{C}([t_1,t_2];\mathbb{R}^n)$  the standard space of continuous functions, and by $\mathcal{C}^\infty_c := \mathcal{C}^\infty_c ([t_1,t_2];\mathbb{R}^n)$ as the standard space of infinitely differentiable functions compactly supported in $(t_1,t_2)\,.$
\end{itemize}
Let us remind that the compact embedding $\mathrm{W}^{1,r} \hookrightarrow \mathcal{C}$ holds for $1 < r \le +\infty$ (see \cite{MR2759829}).

Now, We define  the isoperimetric variational problem under consideration:
\begin{problem}(The isoperimetric variational problem with time delay)
\label{Pb1}
The isoperimetric  problem of the calculus of variations with time delay consists to find the stationary functions of the functional \eqref{Pe}, subject to  isoperimetric equality constraints
\begin{equation}
\label{CT} I^{\tau}[q(\cdot)]=\int_{t_1}^{t_2} g[q]_{\tau}(t)
dt=l,\,\,\,l\in\mathbb{R}\,,
\end{equation}
and  boundary conditions \eqref{Pe2}.

We assume that $l$ is a specified real constant and the functionals $J^{\tau}, I^{\tau}$ are defined in  a weakly closed subset $\mathbb{U}$ of $\mathrm{W}^{1,p}\,.$ We also assume that $(t,s, y, u, v) \mapsto L(t,s, y, u, v)$ and $(t,s, y, u, v) \mapsto g(t,s, y, u, v)$  to be
 a $C^1$-function.
 \end{problem}

\subsection{Isoperimetric Euler--Lagrange equations with time delay}
\label{iso}

Theorem~\ref{th:EL1} motivates the following definition.

\begin{definition} An admissible function $q(\cdot)\in\mathcal{C}^2$  is an extremal for problem \eqref{CT}--\eqref{Pe2}
if it satisfies the following Euler--Lagrange equations with time delay:
\begin{equation}
\label{EL11}
\begin{cases}
\frac{d}{dt}\left\{\partial_{3}g[q]_{\tau}(t)+
\partial_{5}g[q]_{\tau}(t+\tau)\right\}
=\partial_{2}g[q]_{\tau}(t)+\partial_{4}g[q]_{\tau}(t+\tau),
\quad t_{1}\leq t\leq t_{2}-\tau,\\
\frac{d}{dt}\partial_{3}g[q]_{\tau}(t) =\partial_{2}g[q]_{\tau}(t),
\quad t_{2}-\tau< t\leq t_{2}.
\end{cases}
\end{equation}
\end{definition}

The arguments of the calculus of variations assert that by using the Lagrange multiplier rule, Problem~\ref{Pb1} is equivalent to the following augmented problem \cite[$\S12.1$]{CD:Gel:1963}: to minimize
\begin{equation}
\label{agp}
\begin{split}
J^{\tau}[q(\cdot),\lambda]
&= \int_{t_1}^{t_2} F[q,\lambda]_{\tau}(t)dt\\
&:=\int_{t_1}^{t_2}
\left[L[q]_{\tau}(t)
-\lambda \cdot g[q]_{\tau}(t)\right] dt
\end{split}
\end{equation}
subject to \eqref{Pe2}, where $[q,\lambda]_{\tau}(t)=(t,q(t),\dot{q}(t),q(t-\tau),\dot{q}(t-\tau),\lambda)$.

 The augmented Lagrangian
\begin{equation}
\label{eq:aug:Lag}
F:=L-\lambda \cdot g,
\end{equation}
$\lambda\in\mathbb{R}^k$,
has an important role in our study.

The notion of extremizer (a local minimizer or a local maximizer)
can be found in \cite{CD:Gel:1963}. Extremizers can be classified
as normal or abnormal.

\begin{definition}\label{def:extr}
An extremizer of Problem~\ref{Pb1} that does not satisfy \eqref{EL11} is said to be a normal extremizer; otherwise (i.e., if it satisfies \eqref{EL11} for all $t\in[t_1,t_2]$),
is said to be abnormal.
\end{definition}

The following theorem gives a necessary condition for $q(\cdot)$ to be a solution of  Problem~\ref{Pb1} under the assumption that $q(\cdot)$ is
a normal extremizer.

\begin{theorem}
\label{Thm:FractELeq1} If
%$q(\cdot)\in C^2\left([t_{1}-\tau,t_{2}],\mathbb{R}^n\right)$
$q(\cdot)\in C^2\left([t_{1}-\tau,t_{2}]\right)$ is a normal
extremizer to Problem~\ref{Pb1}, then it satisfies the following
\emph{isoperimetric Euler--Lagrange equation with time delay}:
\begin{equation}
\label{eq:eldf11}
\begin{cases}
\frac{d}{dt}\left\{\partial_{3}F[q,\lambda]_{\tau}(t)+
\partial_{5}F[q,\lambda]_{\tau}(t+\tau)\right\}\\
=\partial_{2}F[q,\lambda]_{\tau}(t)+\partial_{4}F[q,\lambda]_{\tau}(t+\tau),
\quad t_{1}\leq t\leq t_{2}-\tau,\\,
\frac{d}{dt}\partial_{3}F[q,\lambda]_{\tau}(t) =\partial_{2}F[q,\lambda]_{\tau}(t), \quad t_{2}-\tau< t\leq t_{2}\,,
\end{cases}
\end{equation}
$t \in [t_1,t_2]$, where $F$ is the augmented Lagrangian \eqref{eq:aug:Lag}
associated with Problem~\ref{Pb1}.
\end{theorem}

\begin{proof}
Consider neighboring functions of the form
\begin{equation}
\label{admfunct} \hat{q}(t)=q(t)+\epsilon_1h_1(t)+\epsilon_2h_2(t),
\end{equation}
where for each $i\in\{1,2\}$ $\epsilon_i$ is a sufficiently small parameter, $h_i$ are assumed to be functions of class
%$C^2\left([t_{1}-\tau,t_{2}],\mathbb{R}^{n}\right)$,
$C^2\left([t_{1}-\tau,t_{2}]\right)$,
 $h_i(t)=0$ for $t\in[t_{1}-\tau,t_{1}]$ and $h_i(t_2)=0$.

First, we will show that (\ref{admfunct}) has a subset of admissible functions for the variational isoperimetric problem with time delay. Consider the quantity
$$
I^{\tau}[\hat{q}(\cdot)]\\=\int_{t_1}^{t_2} g(t,\hat{q}(t),
\dot{\hat{q}}(t),\hat{q}(t-\tau),\dot{\hat{q}}(t-\tau))dt.
$$
Then we can regard $I^{\tau}[\hat{q}(\cdot)]$ as a function of $\epsilon_1$ and $\epsilon_2$. Define
$\hat{I}(\epsilon_1,\epsilon_2)=I^{\tau}[\hat{q}(\cdot)]-l$.
Thus,
\begin{equation}
\label{implicit1}
\hat I(0,0)=0.
\end{equation}
On the other hand, we have
\begin{align*}
\left.\frac{\partial \hat I}{\partial \epsilon_2} \right|_{(0,0)}
&=\int_{t_1}^{t_2}\left[
\partial_2g[q]_{\tau}(t)\cdot h_2(t)+\partial_3g[q]_{\tau}(t)\cdot\dot{h}_2(t)\right]dt\nonumber\\&+
\int_{t_1}^{t_2}\left[
\partial_4g[q]_{\tau}(t)\cdot h_2(t-\tau)+\partial_5g[q]_{\tau}(t)\cdot\dot{h}_2(t-\tau)\right]dt\,.
\end{align*}

Using the change of variable $t = s + \tau$ in the second integral and recalling that $h_2$ is null in $[t_1-\tau, t_1]$, we obtain that
\begin{align}
\left.\frac{\partial \hat I}{\partial \epsilon_2} \right|_{(0,0)}
&=\int_{t_1}^{t_2}\left[
\partial_2g[q]_{\tau}(t)\cdot h_2(t)+\partial_3g[q]_{\tau}(t)\cdot\dot{h}_2(t)\right]dt\nonumber\\&+
\int_{t_1}^{t_2-\tau}\left[
\partial_4g[q]_{\tau}(t+\tau)\cdot h_2(t)+\partial_5g[q]_{\tau}(t+\tau)\cdot\dot{h}_2(t)\right]dt\,.
\label{rui0}
\end{align}

Applying integration by parts and since equation \eqref{rui0} holds
for all admissible variations $h_2$ such that $h_2 =$ 0 for all
$t\in [t_1-\tau, t_1]\,,$ we get
\begin{align*}
\left.\frac{\partial \hat I}{\partial \epsilon_2} \right|_{(0,0)}
&=\int_{t_1}^{t_2-\tau}\left[
\partial_2g[q]_{\tau}(t)-\frac{d}{dt}\partial_3g[q]_{\tau}(t)+\partial_4g[q]_{\tau}(t+\tau)-\frac{d}{dt}\partial_5g[q]_{\tau}(t+\tau)\right]\cdot
h_2(t)dt\\
&+\int^{t_2}_{t_2-\tau}\left[\partial_2g[q]_{\tau}(t)-\frac{d}{dt}\partial_3g[q]_{\tau}(t)\right]\cdot
h_2(t)dt+\partial_5g[q]_{\tau}(t_2)h_2(t_2-\tau).
\end{align*}

Now, if we restrict ourselves to those admissible variations $h_2$
such that $h_2 =$ 0 for all $t\in [t_1, t_2-\tau]\,,$ we obtain
\begin{align*}
\left.\frac{\partial \hat I}{\partial \epsilon_2} \right|_{(0,0)}
&=\int_{t_2-\tau}^{t_2}\left[
\partial_2g[q]_{\tau}(t)-\frac{d}{dt}\partial_3g[q]_{\tau}(t)\right]\cdot h_2(t)dt\,.
\label{rui011}
\end{align*}

Since $q(\cdot)$ is not an extremal for problem \eqref{CT}--\eqref{Pe2}, by the
fundamental lemma of the calculus of variations (see, \textrm{e.g.}, \cite{Bruce:book}), there exists a function $h_2$ such that
\begin{equation}
\label{implicit2}
\left.\frac{\partial \hat I}{\partial \epsilon_2} \right|_{(0,0)}\neq 0.
\end{equation}
Using (\ref{implicit1}) and (\ref{implicit2}), the implicit function
theorem asserts that there exists a function $\epsilon_2(\cdot)$,
defined in a neighborhood of zero, such that $\hat
I(\epsilon_1,\epsilon_2(\epsilon_1))=0$. Consider the real function
$\hat J(\epsilon_1,\epsilon_2)=J^{\tau}[\hat{q}(\cdot)]$. By
hypothesis, $\hat J$ has minimum (or maximum) at $(0,0)$ subject to
the constraint $\hat I(0,0)=0$, and we have proved that $\nabla \hat
I(0,0)\neq \textbf{0}$. Then, we can appeal to the Lagrange multiplier
rule (see, \textrm{e.g.}, \cite[p.~77]{Bruce:book}) to assert the
existence of a number $\lambda$ such that $\nabla(\hat
J(0,0)-\lambda \cdot\hat I(0,0))=\textbf{0}$. Repeating the
calculations as before,

\begin{align*}
\left.\frac{\partial \hat J}{\partial \epsilon_1} \right|_{(0,0)}
&=\int_{t_1}^{t_2-\tau}\left[
\partial_2L[q]_{\tau}(t)-\frac{d}{dt}\partial_3L[q]_{\tau}(t)\right.\\
&+
\left.\partial_4L[q]_{\tau}(t+\tau)-\frac{d}{dt}\partial_5L[q]_{\tau}(t+\tau)\right]\cdot
h_1(t)dt\\
&+
\int_{t_2-\tau}^{t_2}\left[
\partial_2L[q]_{\tau}(t)-\frac{d}{dt}\partial_3L[q]_{\tau}(t)\right]\cdot h_1(t)dt+\partial_5L[q]_{\tau}(t_2)h_1(t_2-\tau),
\end{align*}
and
\begin{align*}
\left.\frac{\partial \hat I}{\partial \epsilon_1} \right|_{(0,0)}
&=\int_{t_1}^{t_2-\tau}\left[
\partial_2g[q]_{\tau}(t)-\frac{d}{dt}\partial_3g[q]_{\tau}(t)\right.\\
&+
\left.\partial_4g[q]_{\tau}(t+\tau)-\frac{d}{dt}\partial_5g[q]_{\tau}(t+\tau)\right]\cdot
h_1(t)dt\\
&+\int_{t_2-\tau}^{t_2}\left[
\partial_2g[q]_{\tau}(t)-\frac{d}{dt}\partial_3g[q]_{\tau}(t)\right]\cdot h_1(t)dt+\partial_5g[q]_{\tau}(t_2)h_1(t_2-\tau)\,.
\end{align*}

Therefore, since $q(t)$ is an extremizer to Problem~\ref{Pb1},  for all admissible variations $h_1$ such that $h_1 =0$  for all $t\in [t_2-\tau, t_2]\,,$ we get
\begin{align}\label{rui01}
&\int_{t_1}^{t_2-\tau}\left[
\partial_2L[q]_{\tau}(t)+
\partial_4L[q]_{\tau}(t+\tau)-\frac{d}{dt}\left(\partial_3L[q]_{\tau}(t)
+\partial_5L[q]_{\tau}(t+\tau)\right)\right.\nonumber\\
&\left.-\lambda\cdot\left(\partial_2g[q]_{\tau}(t)+\partial_4g[q]_{\tau}(t+\tau)
-\frac{d}{dt}\left(\partial_3g[q]_{\tau}(t)
+\partial_5g[q]_{\tau}(t+\tau)\right)\right)\right]\cdot h_1(t)dt=0,
\end{align}
and for all admissible variations $h_1$ such that $h_1 =0$  for all $t\in [t_1, t_2-\tau]\,,$ we get
\begin{align}\label{rui011}
&\int_{t_2-\tau}^{t_2}\left[
\partial_2L[q]_{\tau}(t)-\frac{d}{dt}\partial_3L[q]_{\tau}(t)
-\lambda\cdot\left(\partial_2g[q]_{\tau}(t)
-\frac{d}{dt}\partial_3g[q]_{\tau}(t) \right)\right]\cdot
h_1(t)dt=0\,.
\end{align}

Since equation \eqref{rui01} (and \eqref{rui011}) holds for any
function $h_1$ satisfying $h_1 =0$  for all $t\in [t_2-\tau, t_2]$ (and $h_1 =0$  for all $t\in [t_1, t_2-\tau]$), from the fundamental lemma of the calculus of
variations (see, e.g., [12]), we obtain equations (\ref{eq:eldf11}).
\end{proof}

\begin{remark}
\label{re:EL} If one extends the set of admissible functions in
Problem~\ref{Pb1} to the class of Lipschitz continuous functions,
then the Euler--Lagrange equations (\ref{eq:eldf11}) remain valid
(cf. \cite{GF2013}).
\end{remark}

\begin{definition}[Isoperimetrc extremals with time delay] The solutions
%$q(\cdot)\in C^2\left([t_{1}-\tau,t_{2}],\mathbb{R}^n\right)$
$q(\cdot)\in C^2\left([t_{1}-\tau,t_{2}]\right)$ of the
Euler--Lagrange equations (\ref{eq:eldf11}) are called
\emph{isoperimetric extremals with time delay}.
\end{definition}

\begin{remark} Note that if there is no time delay, that is, if $\tau = 0$, then Problem~\ref{Pb1}
reduces to the classical isoperimetric variational problem:
\begin{gather*}
J[q(\cdot)] =\int_{t_1}^{t_2} L\left(t,q(t),\dot{q}(t)\right) dt
\longrightarrow \min, \label{isocv}\\
\int_{t_1}^{t_2} g\left(t,q(t),\dot{q}(t)\right) dt=l \,.\label{CT1}
\end{gather*}
\end{remark}

\subsection{The DuBois–-Reymond necessary condition}

\label{iso1}
The following theorem gives a generalization of the
DuBois--Reymond necessary condition for classical variational
problems [4] and generalizes the Dubois--Reymond necessary condition
for isoperimetric variational problems with time delay of [11].

\begin{theorem}[Isoperimetric DuBois--Reymond necessary condition with time delay]
\label{theo:cdrnd} If $q(\cdot)$ is an isoperimetric extremals with time delay such that
\begin{equation}\label{CDUR}
\partial_4F[q]_{\tau}(t+\tau)\cdot \dot{q}(t)+\partial_5F[q]_{\tau}(t+\tau)\cdot
\ddot{q}(t)=0
\end{equation}
for all $t\in[t_1-\tau,t_2-\tau]$, then it satisfies the following
conditions:
\begin{equation}
\label{eq:cdrnd}
\frac{d}{dt}\left\{F[q]_{\tau}(t)-\dot{q}(t)\cdot(\partial_{3}
F[q]_{\tau}(t) +\partial_{5} F[q]_{\tau}(t+\tau))\right\} =
\partial_{1} F[q]_{\tau}(t)
\end{equation}
for $t_1\leq t\leq t_{2}-\tau$, and
\begin{equation}
\label{eq:cdrnd1} \frac{d}{dt}\left\{F[q]_{\tau}(t)
-\dot{q}(t)\cdot\partial_{3} F[q]_{\tau}(t)\right\} =\partial_{1}
F[q]_{\tau}(t)
\end{equation}
for $t_2-\tau< t\leq t_{2}$, where $F$ is defined in
\eqref{eq:aug:Lag}.
\end{theorem}

\begin{proof}
We only prove the theorem in the interval $t_{1}\leq t\leq
t_{2}-\tau$ (the proof is similar in the interval $t_{2}-\tau<
t\leq t_{2}$). We derive equation \eqref{eq:cdrnd} as follows:

Let an arbitrary  $x\in [t_1, t_2-\tau]\,.$ Note that
\begin{equation}
\label{pro}
\begin{split}
\int_{t_1}^{x}\frac{d}{dt}&\left[F[q]_{\tau}(t)-\dot{q}(t)\cdot(\partial_{3}
F[q]_{\tau}(t)
+\partial_{5} F[q]_{\tau}(t+\tau))\right]dt\\
&=\int_{t_1}^{x}\Bigr[\partial_1 \left(L[q]_{\tau}(t)-\lambda\cdot
g[q]_{\tau}(t)\right)+
\partial_2 \left(L[q]_{\tau}(t)-\lambda\cdot
g[q]_{\tau}(t)\right)\cdot \dot{q}(t)\\
&-\partial_{5} \left(L[q]_{\tau}(t+\tau)-\lambda\cdot
g[q]_{\tau}(t+\tau)\right)\cdot\ddot{q}(t)\\
&\qquad -\frac{d}{dt}\left\{\partial_{3}
\left(L[q]_{\tau}(t)-\lambda\cdot g[q]_{\tau}(t)\right)
+\partial_{5} \left(L[q]_{\tau}(t+\tau)-\lambda\cdot
g[q]_{\tau}(t+\tau)\right)\right\}\cdot \dot{q}(t)\Bigr]dt\\
&\qquad +\int_{t_1}^{x}\left[\partial_4
\left(L[q]_{\tau}(t)-\lambda\cdot
g[q]_{\tau}(t)\right)\cdot\dot{q}(t-\tau)\right.\\
&\left. +\partial_5 \left(L[q]_{\tau}(t)-\lambda\cdot
g[q]_{\tau}(t)\right)\cdot\ddot{q}(t-\tau)\right]dt.
\end{split}
\end{equation}
Observe that, by hypothesis \eqref{CDUR}, the last integral of
\eqref{pro} is null and by substituting the Euler--Lagrange equation
with time delay \eqref{eq:eldf11}, the equation \eqref{pro} becomes
\begin{equation*}
\label{eq3}
\begin{split}
\int_{t_1}^{x}\frac{d}{dt}&\left[F[q]_{\tau}(t)-\dot{q}(t)\cdot(\partial_{3}
F[q]_{\tau}(t)
+\partial_{5} F[q]_{\tau}(t+\tau))\right]dt\\
&=\int_{t_1}^{x}\Bigl(\partial_1
\left(L[q]_{\tau}(t)-\lambda\cdot g[q]_{\tau}(t)\right)\\
& -\left[\partial_4 \left(L[q]_{\tau}(t)-\lambda\cdot
g[q]_{\tau}(t)\right)\cdot\dot{q}(t) +\partial_5
\left(L[q]_{\tau}(t+\tau)-\lambda\cdot
g[q]_{\tau}(t+\tau)\right)\cdot\ddot{q}(t)\right]\Bigr)dt.
\end{split}
\end{equation*}
Using hypothesis \eqref{CDUR} in the right-hand side of the last equation, we conclude that
\begin{multline}
\label{eq31}
\int_{t_1}^{x}\frac{d}{dt}\left[F[q]_{\tau}(t)-\dot{q}(t)\cdot(\partial_{3}
F[q]_{\tau}(t) +\partial_{5} F[q]_{\tau}(t+\tau))\right]dt\\
=\int_{t_1}^{x}\Bigl(\partial_1 \left(L[q]_{\tau}(t)-\lambda\cdot
g[q]_{\tau}(t)\right)\Bigr)dt\,.
\end{multline}

We finally obtain \eqref{eq:cdrnd} by the arbitrariness
$x\in[t_{1},t_{2}-\tau]\,.$
\end{proof}

\begin{remark}
If we assume that admissible functions in Problem~\ref{Pb1} are Lipschitz continuous, then one can show that the DuBois--Reymond necessary conditions with time delay \eqref{eq:cdrnd}--\eqref{eq:cdrnd1} are still valid (cf. \cite{GF2013}).
\end{remark}

\subsection{Existence and Regularity of a minimizer}

In this section, we prove a theorem, analogous to the classical Tonelli theorem, ensuring the existence of a minimizer for  Problem~\ref{Pb1}, and we will show that full regularity holds for certain variational integrals of the form $J^{\tau}(q,\lambda)\,.$

For this, we need the following definitions:
\begin{definition}
We say that $J^{\tau}(q,\lambda)$  is \emph{coercive} on $\mathbb{U}$ if for any $\lambda\in \mathbb{R}^k$
\begin{equation*}
\lim\limits_{\substack{\Vert q \Vert_{\mathrm{W}^{1,p}}
\to \infty \\ q \in \mathbb{U} }} J^{\tau} (q,\lambda) = +\infty\,.
\end{equation*}
\end{definition}

\begin{definition}
We say that $F$ is \emph{regular} if
\begin{itemize}
\item $F[q]_{\tau}(t) \in \mathrm{L}^1$;
\item $\partial_2 F[q]_{\tau}(t) \in \mathrm{L}^1$;
\item $\partial_3 F[q]_{\tau}(t) \in \mathrm{L}^{p'}$;
\item $\partial_4F[q]_{\tau}(t) \in \mathrm{L}^{1}$;
\item $\partial_5 F[q]_{\tau}(t) \in \mathrm{L}^{p'}$;
\end{itemize}
for any $q\in\mathrm{W}^{1,p}$.
\end{definition}

Next, we state a Tonelli-type theorem for Lagrangian functionals containing time delay. For this to happen we adopt an extension of the method used in \cite{MR0688142,MR2361288}.

\begin{theorem}[Tonelli's existence theorem for isoperimetric variational problems with time delay]
\label{thmtonelli}
Let us assume the following hypotheses:
\begin{itemize}
\item ($H_1$) $F[q]_{\tau}(t)$ is regular;
\item ($H_2$) $J^{\tau}(q,\lambda)$ is coercive on $\mathbb{U}$;
\item ($H_3$) $F[q]_{\tau}(t)$ is convex on $(\mathbb{R}^n)^4$ for any $t \in [t_1,t_2]\,.$
\end{itemize}
Then there exists a minimizer for Problem~\ref{Pb1}.
\end{theorem}

\begin{proof}
$J^{\tau}(q,\lambda)$ is well-defined because, since $F[q]_{\tau}(t)$ is regular,
$F[q]_{\tau}(t) \in \mathrm{L}^1$
and $J^{\tau}(q,\lambda)$ exists in $\mathbb{R}$.
Let $(q_n)_{n \in \mathbb{N}} \subset \mathbb{U}$
be a minimizing sequence satisfying
\begin{equation}
\label{eq:1}
J^{\tau} (q_n,\lambda) \longrightarrow \inf\limits_{q \in \mathbb{U}} J^{\tau}(q,\lambda) < +\infty.
\end{equation}
($H_2$) implies that $(q_n)_{n \in \mathbb{N}}$ is bounded  in $\mathrm{W}^{1,p}\,.$ Since $\mathrm{W}^{1,p}$ is a reflexive Banach space, there exists  a subsequence of $(q_n)_{n \in \mathbb{N}}$ weakly convergent to $\tilde{q}$ in $\mathrm{W}^{1,p}\,.$ We denote this subsequence by $(\tilde{q}_n)_{n \in \mathbb{N}}\,.$
 Furthermore, since $\mathbb{U}$ is a weakly closed subset of $\mathrm{W}^{1,p}$, $\tilde{q} \in \mathbb{U}$.

Now, from ($H_3$) we have
\begin{multline}
\label{eq00}
J^{\tau}(\tilde{q}_n,\lambda) \geq J^{\tau}(\tilde{q},\lambda) +  \int_{t_1}^{t_2}\left[\partial_2 F[\tilde{q}]_{\tau}(t) \cdot (\tilde{q}_n-\tilde{q})+
\partial_3F[\tilde{q}]_{\tau}(t) \cdot (\dot{\tilde{q}}_n-\dot{\tilde{q}})\right.\\ \left. + \partial_4 F[\tilde{q}]_{\tau}(t) \cdot (\tilde{q}_n-\tilde{q})+
\partial_5F[\tilde{q}]_{\tau}(t) \cdot (\dot{\tilde{q}}_n-\dot{\tilde{q}})\right]\, dt\,,
\end{multline}
and from the following assumptions:
\begin{enumerate}
  \item $F[\tilde{q}]_{\tau}(t)$ is regular;
  \item $\tilde{q}_n \longrightarrow \tilde{q}$  in ${\mathrm{W}^{1,p}} $;
  \item the compact embedding $\mathrm{W}^{1,p} \hookrightarrow \mathcal{C}$ holds;
\end{enumerate}
one can conclude that
\begin{description}
  \item[(i)] $\partial_2 F[\tilde{q}]_{\tau}(t) \in \mathrm{L}^1$ and $\tilde{q}_n \longrightarrow \tilde{q}$ in $\mathrm{L}^{\infty}\,;$
  \item[(ii)] $\partial_3 F[\tilde{q}]_{\tau}(t) \in \mathrm{L}^{p'}$ and $\dot{\tilde{q}}_n \rightharpoonup \dot{\tilde{q}}$ in $\mathrm{L}^{p}\,;$
  \item[(iii)] $\partial_4 F[\tilde{q}]_{\tau}(t) \in \mathrm{L}^1$ and $\tilde{q}_n \longrightarrow \tilde{q}$ in $\mathrm{L}^{\infty}\,;$
  \item[(iv)] $\partial_5 F[\tilde{q}]_{\tau}(t) \in \mathrm{L}^{p'}$ and $\dot{\tilde{q}}_n \rightharpoonup \dot{\tilde{q}}$ in $\mathrm{L}^{p}\,.$
\end{description}
To complete the proof, we use  \eqref{eq:1} and we take $n \to \infty$ in inequality \eqref{eq00}, to obtain
\begin{equation*}
\inf\limits_{q \in \mathbb{U}} J^{\tau}(q,\lambda) \geq J^{\tau}(\tilde{q},\lambda) \in \mathbb{R}\,.
\end{equation*}
\end{proof}

\begin{theorem}[Regualarity Theorem]\label{th:RE} Let $I=(t_1-\tau,t_2)$ be a bounded interval in $\mathbb{R}$, and let $F(t,z,v,z_{\tau},v_{\tau})$ be a Lagrngian of class $\mathcal{C}^2$ defined in
$\bar{I}\times\mathbb{R}^n\times\mathbb{R}^n\times\mathbb{R}^n\times\mathbb{R}^n$, $n\geq 1$, satisfying the following conditions:
\begin{enumerate}
 \item there are constants $c_{0}, c_{1}> 0$ such that for all $(t,z,v,z_{\tau},v_{\tau})\in \bar{I}\times\mathbb{R}^n\times\mathbb{R}^n\times\mathbb{R}^n\times\mathbb{R}^n$
\begin{equation}\label{r1}
 c_{0}\left(\parallel v\parallel^m+\parallel v_\tau\parallel^m \right)\leq F(t,z,v,z_{\tau},v_{\tau})\leq c_1\left(1+\parallel v\parallel^m+\parallel v_\tau\parallel^m\right),\,m\geq1\,;
\end{equation}
\item there is a function $M(R)>0$ such that
\begin{multline}\label{r2}
\parallel \partial_2F(t,z,v,z_{\tau},v_{\tau})\parallel+\parallel \partial_3F(t,z,v,z_{\tau},v_{\tau})\parallel\\
+\parallel \partial_4F(t,z,v,z_{\tau},v_{\tau})\parallel+\parallel \partial_5F(t,z,v,z_{\tau},v_{\tau})\parallel\leq M(R)\left(1+\parallel v\parallel^2+\parallel v_\tau\parallel^2\right)\,;
\end{multline}
for all $(t,z,v,z_{\tau},v_{\tau})\in \bar{I}\times\mathbb{R}^n\times\mathbb{R}^n\times\mathbb{R}^n\times\mathbb{R}^n$ with $t^2+\parallel z\parallel^2
+\parallel z_\tau\parallel^2\leq R^2\,;$
\item \begin{equation}\label{r8}
\partial_{33}F(t,z,v,z_{\tau},v_{\tau}),\partial_{55}F(t,z,v,z_{\tau},v_{\tau})>0
\end{equation}  for all $(t,z,v,z_{\tau},v_{\tau})\in \bar{I}\times\mathbb{R}^n\times\mathbb{R}^n\times\mathbb{R}^n\times\mathbb{R}^n\,.$
\end{enumerate}

Let $$\mathfrak{C}:=\left\{q\in \mathrm{W}^{1,m}\left(I,\mathbb{R}^n\right)\,subject\, to\, boundary\, conditions~\eqref{Pe2}\right\}\,.$$
 Suppose that $q$ is a local minimizer of the variational integral
$$J^\tau(q,\lambda)=\int_IF[q,\lambda]_{\tau}(t)dt$$
in $\mathfrak{C}\,.$ Then $q$ belongs to $\mathcal{C}^2\left(\bar{I}\right)$ and satisfies the Euler--Lagrange equations~\eqref{eq:eldf11} on $I\,.$
\end{theorem}

\begin{proof}

We will prove this theorem in four steps:
\begin{itemize}
  \item $\underline{Step\,1}\,:$ $q$ is a weak $W^{1,m}$-isoperimetric extremal.

  First, note that $J^\tau[(q,\lambda]$ is well-defined. Indeed, for all $\nu\in\mathfrak{C}$ and by~\eqref{r1}, we have
  \begin{multline*}
  c_0\int_I\left(\parallel \dot{\nu}(t)\parallel^m+\parallel \dot{\nu}(t-\tau)\parallel^m \right)dt\leq J^\tau[(q,\lambda]\\
  \leq c_1\int_I\left(1+\parallel \dot{\nu}(t)\parallel^m+\parallel \dot{\nu}(t-\tau)\parallel^m \right)dt<\infty\,.
  \end{multline*}

  Let a function $h\in Lip(I,\mathbb{R}^n)$ such that $\parallel h(t)\parallel\leq k$ on $I$, and $\parallel \dot{h}(t)\parallel\leq k$ a.e. on  $I$ for some constant $k>0\,.$ Let $0<\varepsilon_0\leq 1\,,$
  and let $\varepsilon \in \mathbb{R}$ such that $\mid \varepsilon\mid <\varepsilon_0\,.$ The generalized version of the fundamental theorem of calculus asserts that
  \begin{multline*}
  t^2+ \left(q(t) +\varepsilon h(t)\right)^2+\left(q(t-\tau) +\varepsilon h(t-\tau)\right)^2\leq t_2^2+ \left(q(t) +\varepsilon h(t)\right)^2+\left(q(t-\tau) +\varepsilon h(t-\tau)\right)^2\\
  \leq t_2^2+\left[\int_I \parallel\dot{q}(t) \parallel dt+k\right]^2+\left[\int_I \parallel\dot{q}(t-\tau) \parallel dt+k\right]^2<\infty
  \end{multline*}
  for all $t \in \bar{I}\,.$ So there exists $R \in \mathbb{R}$ such that
  \begin{equation*}
    t^2+ \left(q(t) +\varepsilon h(t)\right)^2+\left(q(t-\tau) +\varepsilon h(t-\tau)\right)^2\leq R\,.
  \end{equation*}
  Furthermore,
  \begin{equation*}
    \parallel \dot{q}(t) +\varepsilon \dot{h}(t)\parallel^m +\parallel \dot{q}(t-\tau) +\varepsilon \dot{h}(t-\tau)\parallel^m
    \leq 2^{m-1}\left(\parallel \dot{q}(t)\parallel^m+\parallel \dot{q}(t-\tau)\parallel^m+(2k)^m\right)\,.
  \end{equation*}
  On account of~\eqref{r2} we see that the function
   \begin{multline*}
     \partial_2F(t,q(t)+\varepsilon h(t),\dot{q}(t)+\varepsilon \dot{h}(t),q(t-\tau)+\varepsilon h(t-\tau),\dot{q}(t-\tau)+\varepsilon \dot{h}(t-\tau))\\
    + \partial_3F(t,q(t)+\varepsilon h(t),\dot{q}(t)+\varepsilon \dot{h}(t),q(t-\tau)+\varepsilon h(t-\tau),\dot{q}(t-\tau)+\varepsilon \dot{h}(t-\tau))\\
     + \partial_4F(t,q(t)+\varepsilon h(t),\dot{q}(t)+\varepsilon \dot{h}(t),q(t-\tau)+\varepsilon h(t-\tau),\dot{q}(t-\tau)+\varepsilon \dot{h}(t-\tau))\\
      + \partial_5F(t,q(t)+\varepsilon h(t),\dot{q}(t)+\varepsilon \dot{h}(t),q(t-\tau)+\varepsilon h(t-\tau),\dot{q}(t-\tau)+\varepsilon \dot{h}(t-\tau))
 \end{multline*}
 is a.e. dominated by the $\mathrm{L}^1$-function
 \begin{equation*}
   kM(R)\left[1+2^{m-1}\left(\parallel \dot{q}(t)\parallel^m+\parallel \dot{q}(t-\tau)\parallel^m+(2k)^m\right)\right]\,.
 \end{equation*}
 Then Lebesgue's dominated convergence theorem yields that the function $\Phi(\varepsilon):=J^\tau(q+\varepsilon h,\lambda)$ is  of class $\mathcal{C}^1$ on $(-\varepsilon_0,\varepsilon_0)\,.$ Since $q$ is a local minimizer of $J^\tau[(q,\lambda]$ in $\mathfrak{C}$, it follows that $\Phi(0)\leq \Phi(\varepsilon)$  for $\mid\varepsilon\mid<\varepsilon_0$ if $h\in\mathcal{C}_c^\infty\left(I,\mathbb{R}^n\right)$,  whence $\Phi(0)=0$ and therefore
 \begin{multline}\label{r3}
     \int_I\left\{\partial_2F(t,q(t),\dot{q}(t),q(t-\tau),\dot{q}(t-\tau))\cdot h(t)+ \partial_3F(t,q(t),\dot{q}(t),q(t-\tau),\dot{q}(t-\tau))\cdot \dot{h}(t)\right.\\\left.
    + \partial_4F(t,q(t),\dot{q}(t),q(t-\tau),\dot{q}(t-\tau))\cdot h(t-\tau)\right.\\\left.
    + \partial_5F(t,q(t),\dot{q}(t),q(t-\tau),\dot{q}(t-\tau))\cdot \dot{h}(t-\tau)\right\}dt=0
 \end{multline}
for all $h\in\mathcal{C}_c^\infty\left(I,\mathbb{R}^n\right)\,.$ Equation~\eqref{r3} is the formula of the first variation of $J^\tau[(q,\lambda]$, it follows
that $q\in AC\left(I,\mathbb{R}^n\right)$ with $\dot{q}\in \mathrm{L}^m\left(I,\mathbb{R}^n\right)$ is a weak $\mathrm{W}^{1,m}$-extremal.

  \item $\underline{Step\,2}\,:$ $q\in \mathcal{C}^1\left(\bar{I},\mathbb{R}^n\right)$.

  We can show that equation~\eqref{r3} is equivalent to (see the proof of Theorem~\ref{Thm:FractELeq1})
  \begin{align}
0
&=\int_{t_1}^{t_2}\left[
\partial_2F[q,\lambda]_{\tau}(t)\cdot h(t)+\partial_3F[q,\lambda]_{\tau}(t)\cdot\dot{h}(t)\right]dt\nonumber\\&+
\int_{t_1}^{t_2-\tau}\left[
\partial_4F[q,\lambda]_{\tau}(t+\tau)\cdot h(t)+\partial_5F[q,\lambda]_{\tau}(t+\tau)\cdot\dot{h}(t)\right]dt\,.
\label{r5}
\end{align}

 From now on we only focus  in the interval $t_{2}-\tau<
t\leq t_{2}$ (the proof is similar in the interval $t_{1}\leq t\leq
t_{2}-\tau$).
  By virtue of~\eqref{r2}, we have $$\partial_2F[q,\lambda]_{\tau}(t)\in \mathrm{L}^1\left((t_2-\tau,t_2),\mathbb{R}^n\right)\,.$$
 On account the boundary conditions of $h$ (the same as in the proof of Theorem~\ref{Thm:FractELeq1}), an integration by parts leads to
  \begin{multline}\label{r6}
    \int_{t_{2}-\tau}^{t_2}\partial_2F[q,\lambda]_{\tau}(t)\cdot h(t)dt=\left[h(t)\cdot\int_{t_{2}-\tau}^t\partial_2F[q,\lambda]_{\tau}(x)dx \right]_{t_2-\tau}^{t_2}\\
    - \int_{t_{2}-\tau}^{t_2}\left(\int_{t_{2}-\tau}^t\partial_2F[q,\lambda]_{\tau}(x)dx\right)\cdot \dot{h}(t)dt= - \int_{t_{2}-\tau}^{t_2}\left(\int_{t_{2}-\tau}^t\partial_2F[q,\lambda]_{\tau}(x)dx\right)\cdot \dot{h}(t)dt
  \end{multline}
  for all $h\in\mathcal{C}_c^\infty\left(I,\mathbb{R}^n\right)\,.$
  Combining~\eqref{r6} with ~\eqref{r5} in the interval $t_{2}-\tau<
t\leq t_{2}$, we get
\begin{equation*}
  \int_{t_{2}-\tau}^{t_2}\left[\partial_3F[q,\lambda]_{\tau}(t)
  - \int_{t_{2}-\tau}^t\partial_2F[q,\lambda]_{\tau}(x)dx\right]\cdot \dot{h}(t)dt=0\,.
\end{equation*}
Then we can apply the DuBois--Reymond's Lemma (see  \cite{Bot} Section $(1.1)$, Propositions $(1.9)$), to show that there
is a vector $c\in\mathbb{R}^n$ such that
\begin{equation}\label{r7}
 \partial_3F[q,\lambda]_{\tau}(t)=c+\int_{t_{2}-\tau}^t\partial_2F[q,\lambda]_{\tau}(x)dx\qquad a.e.\,\,on\,\, (t_{2}-\tau,t_2)\,.
\end{equation}

By virtue of~\eqref{r2}, we have $$\partial_3F[q,\lambda]_{\tau}(t)\in \mathrm{L}^1\left((t_2-\tau,t_2),\mathbb{R}^n\right)\,.$$

 Consider the mapping
 $$\Psi(t,z,v,z_\tau,v_\tau)=(t,z,\partial_3F(t,z,v,z_{\tau},v_{\tau}),z_\tau,\partial_5F(t,z,v,z_{\tau},v_{\tau}))\,.$$

 Because condition~\eqref{r8}, we have that $$D\Psi(t,z,v,z_{\tau},v_{\tau})\neq0$$ for all
  $$(t,z,v,z_{\tau},v_{\tau})\in \left(\overline{(t_2-\tau,t_2)}\times\mathbb{R}^n\times\mathbb{R}^n\times\mathbb{R}^n\times\mathbb{R}^n\right)\,,$$
  so $\Psi$ is injective. By the Global Inverse Function Theorem (see for example \cite{CD:kolk:2004}), we have that $$\Psi:\left(\overline{(t_2-\tau,t_2)}\times\mathbb{R}^n\times\mathbb{R}^n\times\mathbb{R}^n\times\mathbb{R}^n\right)\rightarrow \Psi\left(\overline{(t_2-\tau,t_2)}\times\mathbb{R}^n\times\mathbb{R}^n\times\mathbb{R}^n\times\mathbb{R}^n\right)$$
  is a $\mathcal{C}^1$-diffeomorphism. Moreover, condition~\eqref{r8}   implies that $$\partial_3F(t,z,v,z_{\tau},v_{\tau}),\partial_5F(t,z,v,z_{\tau},v_{\tau})$$ are injective, and thus
  $$\Psi\left(\overline{(t_2-\tau,t_2)}\times\mathbb{R}^n\times\mathbb{R}^n\times\mathbb{R}^n\times\mathbb{R}^n\right)= \overline{(t_2-\tau,t_2)}\times\mathbb{R}^n\times\mathbb{R}^n\times\mathbb{R}^n\times\mathbb{R}^n\,.$$

  Define
  \begin{gather*}
  \pi(t)=c+\int_{t_{2}-\tau}^t\partial_2F[q,\lambda]_{\tau}(x)dx\,;\\
  \pi_1(t)=c_3+\int_{t_{2}-\tau}^t\partial_4F[q,\lambda]_{\tau}(x)dx,\,\,c_3\in\mathbb{R}^n\,;\\
  \sigma(t)=(t,q(t),\dot{q}(t),q(t-\tau),\dot{q}(t-\tau))\,;\\
  \varrho(t)=(t,q(t),\pi(t),q(t-\tau),\pi_1(t))\,.
  \end{gather*}
  Then $\sigma$ is defined a.e. on $(t_2-\tau,t_2)$, whereas $\varrho$ is defined  for all $t\in \overline{(t_2-\tau,t_2)}\,.$ Moreover,
  \begin{equation}\label{r10}
    \Psi(\sigma(t))=\varrho(t)\quad a.e.\quad on\,\, (t_2-\tau,t_2)\,.
  \end{equation}
  The image set $\varrho\left(\overline{(t_2-\tau,t_2)}\right)$ lies the range of $\Psi$, and $\varrho$ is continuous on $\overline{(t_2-\tau,t_2)}\,.$
  Thus the function
  $$(t,q(t),\nu(t),q(t-\tau),\nu_1(t))=\Psi^{-1}(\varrho(t))\,,\qquad t\in \overline{(t_2-\tau,t_2)}\,,$$
  is well-defined and continuous.
  On the other hand, \eqref{r10} implies that
 $$(t,q(t),\dot{q}(t),q(t-\tau),\dot{q}(t-\tau))=\sigma(t)=\Psi^{-1}(\varrho(t))\quad a.e.\quad on\,\,(t_2-\tau,t_2)\,,$$
 and therefore $\dot{q}(t)=\nu(t)$ and $\dot{q}(t-\tau)=\nu_1(t)$  a.e. on $(t_2-\tau,t_2)\,.$
 By a generalization of the fundamental formula of calculus, we have that
 $$q(t)=q(t_2-\tau)+\int_{t_2-\tau}^{t}\dot{q}(x) dx=q(t_2-\tau)+\int_{t_2-\tau}^{t}\nu(x)dx\,,$$
 and
 $$q(t-\tau)=q(t_2-2\tau)+\int_{t_2-\tau}^{t}\dot{q}(x-\tau) dx=q(t_2-2\tau)+\int_{t_2-\tau}^{t}\nu_1(x)dx\,,$$
it follows that $q\in \mathcal{C}^1\left(\overline{(t_2-\tau,t_2)},\mathbb{R}^n\right)\,.$

 \item $\underline{Step\,\,3}\,:$ $q\in \mathcal{C}^2\left(\bar{I},\mathbb{R}^n\right)$.
 
 Because $q\in \mathcal{C}^1\left(\bar{I},\mathbb{R}^n\right)$, we get that
 \begin{equation*}
   \partial_3F[q,\lambda]_{\tau}(t)=\pi(t) \quad for \,\,all\,\,t\in(t_2-\tau,t_2)\,.
 \end{equation*}
 Thus the mapping $\Upsilon(t,v):\overline{(t_2-\tau,t_2)}\times\mathbb{R}^n\longrightarrow\mathbb{R}^n$ defined by
 $$\Upsilon(t,v):=\partial_3F[q,\lambda]_{\tau}(t)-\pi(t)$$
 is  of class $\mathcal{C}^1\left(\overline{(t_2-\tau,t_2)}\times\mathbb{R}^n,\mathbb{R}^n\right)$ because $F$ is of class $\mathcal{C}^2\,.$
 Using condition~\eqref{r8}, we obtain
 $$\frac{\partial\Upsilon}{\partial\dot{q}}=\partial_{33}F>0\,.$$
 for all $t\in\overline{(t_2-\tau,t_2)}\,.$

 Since $v=\dot{q}(t)$, $t\in\overline{(t_2-\tau,t_2)}$, is solution of $\Upsilon(t,v)=0$, the implicit function theorem yields that $\dot{q}\in \mathcal{C}^1\left(\overline{(t_2-\tau,t_2)}\right)$, i.e. $q\in \mathcal{C}^2\left(\overline{(t_2-\tau,t_2)}\right)\,.$

  \item $\underline{Step\,\,4}\,:$ $q$ satisfies the Euler--Lagrange equations~\eqref{eq:eldf11}.

  Since  $q\in \mathcal{C}^2\left(\overline{(t_2-\tau,t_2)}\right)$, we can integrate~\eqref{r5} by parts on $(t_2-\tau,t_2)\,.$ From this we get that
  \begin{equation*}
  \int^{t_2}_{t_2-\tau}\left[-\frac{d}{dt}\partial_{3}F[q,\lambda]_{\tau}(t) +\partial_{2}F[q,\lambda]_{\tau}(t)\right]\cdot h(t)dt
  \end{equation*}
  for all $h\in\mathcal{C}_c^\infty\left((t_2-\tau,t_2),\mathbb{R}^n\right)\,.$ Applying the fundamental lemma of calculus of variations, we obtain
  \begin{equation*}
  -\frac{d}{dt}\partial_{3}F[q,\lambda]_{\tau}(t) +\partial_{2}F[q,\lambda]_{\tau}(t)=0\quad on\,\,(t_2-\tau,t_2)\,.
  \end{equation*}
\end{itemize}
This completes the proof.
\end{proof}

\begin{example} The most classical examples of a Lagrangian are the quadratic ones. Consider the following isoperimetric problem of calculus of variations with time delay:
\begin{equation}
\label{eq:ex22}
\begin{gathered}
J^{\tau}[q(\cdot)]=\int_{t_1}^{t_2}\left(\|q(t)\|^2+\|q(t-\tau)\|^2+\|\dot{q}(t)\|^2+\|\dot{q}(t-\tau)\|^2\right)dt \longrightarrow \min,\\
q(t_2)=q_2,
\end{gathered}
\end{equation}
subject to  isoperimetric equality constraints
\begin{equation}
\label{eq:ex23}
 I^{\tau}[q(\cdot)]=\int_{t_1}^{t_2}\left(
\|\dot{q}(t)\|^2+\|\dot{q}(t-\tau)\|^2\right)dt=l
\end{equation}
in the class of functions
%$q(\cdot)\in Lip\left([-1,3];\mathbb{R}\right)$
$q(\cdot)\in  W^{1,2}[t_1-\tau,t_2]$.

For this example, the augmented
Lagrangian $F$ is given as
\begin{equation}\label{exm}
    F=\|q(t)\|^2+\|q(t-\tau)\|^2+\|\dot{q}(t)\|^2+\|\dot{q}(t-\tau)\|^2
    -\lambda\cdot\left(\|\dot{q}(t)\|^2+\|\dot{q}(t-\tau)\|^2\right) \,.
\end{equation}

It is not difficult to verify that the augmented Lagrangian \eqref{exm}  is convex and satisfies the hypotheses of the Theorems~\ref{thmtonelli} and ~\ref{th:RE},
and therefore exists a solution $q(\cdot)\in  [t_1-\tau,t_2]$ of the problem \eqref{eq:ex22}--\eqref{eq:ex23} for some suitable values of $\lambda$ and it satisfies the Euler--Lagrange equations~\eqref{eq:eldf11}.

\end{example}

\subsection{Variational isoperimetric Noether's conservation laws with time delay}
\label{sec:example}

In \cite{GF2013} the authors remark that when one extends Noether's theorem to the biggest class of functions for which one can derive the Euler--Lagrange equations, i.e., for Lipschitz continuous functions, then one can find Lipschitz Euler--Lagrange extremals that fail to satisfy the Noether conserved quantity established in \cite{GF2012}.
They show that to formulate a Noether's theorem with time delays for nonsmooth functions it is enough to restrict the set of delayed Euler--Lagrange extremals to those that satisfy the delayed DuBois--Reymond condition.

The notion of invariance given in Definition~\ref{def:invnd} can be extended up to an exact differential.

\begin{definition}[Invariance up to a gauge-term]
\label{def:invndLIP} We say that the functional \eqref{agp} is invariant
under the $s$-parameter group of infinitesimal transformations
\eqref{eq:tinf} up to the gauge-term $\Phi$ if
\begin{multline}
\label{eq:invndLIP} \int_{I} \dot{\Phi}[q]_{\tau}(t)dt =
\frac{d}{ds} \int_{\bar{t}(I)}
F\left(t+s\eta(t,q(t))+o(s),q(t)+s\xi(t,q(t))+o(s),
\frac{\dot{q}(t)+s\dot{\xi}(t,q(t))}{1+s\dot{\eta}(t,q(t))},\right.\\
\left. q(t-\tau)+s\xi(t-\tau,q(t-\tau))+o(s),\frac{\dot{q}(t-\tau)
+s\dot{\xi}(t-\tau,q(t-\tau))}{1+s\dot{\eta}(t-\tau,q(t-\tau))}\right)
(1+s\dot{\eta}(t,q(t))) dt\Biggr|_{s=0}
\end{multline}
for any  subinterval $I \subseteq [t_1,t_2]$ and for all
%$q(\cdot)\in Lip\left([t_1-\tau,t_2];\mathbb{R}^n\right)\,.$
$q(\cdot)\in Lip\left([t_1-\tau,t_2]\right)\,.$
\end{definition}

\begin{lemma}[Necessary condition of invariance]
\label{thm:CNSI:SCV} If functional \eqref{agp} is invariant up to
$\Phi$ in the sense of Definition~\ref{def:invndLIP}, then
\begin{multline}
\label{eq:cnsind1}
\int_{t_1}^{t_2-\tau}\Bigl[-\dot{\Phi}[q]_{\tau}(t)+\partial_{1}
F[q]_{\tau}(t)\eta(t,q) +\left(\partial_{2}
F[q]_{\tau}(t)+\partial_4 F[q]_{\tau}(t+\tau)\right)\cdot\xi(t,q)\\
+\left(\partial_{3}F[q]_{\tau}(t)
+\partial_5F[q]_{\tau}(t+\tau)\right)\cdot\left(\dot{\xi}(t,q)
-\dot{q}(t)\dot{\eta}(t,q)\right) +
F[q]_{\tau}(t)\dot{\eta}(t,q)\Bigr]dt = 0
\end{multline}
for $t_1\leq t\leq t_{2}-\tau$ and
\begin{multline}
\label{eq:cnsind2}
\int_{t_2-\tau}^{t_2}\Bigl[-\dot{\Phi}[q]_{\tau}(t)
+\partial_{1}F[q]_{\tau}(t)\eta(t,q)
+\partial_{2}F[q]_{\tau}(t)\cdot\xi(t,q)\\
+\partial_{3}F[q]_{\tau}(t)\cdot\left(\dot{\xi}(t,q)
-\dot{q}(t)\dot{\eta}(t,q)\right)+F[q]_{\tau}(t)\dot{\eta}(t,q)\Bigr]dt
=0
\end{multline}
for $t_2-\tau\leq t\leq t_{2}$.
\end{lemma}

\begin{proof}
Without loss of generality, we take $I=[t_1,t_2]$.
Then, \eqref{eq:invndLIP} is equivalent to
\begin{equation}
\label{eq:cnsind3}
\begin{split}
\int_{t_1}^{t_2} \Bigl[ &-\dot{\Phi}[q]_{\tau}(t)
+\partial_{1}\left(L[q]_{\tau}(t)-\lambda\cdot g[q]_{\tau}(t)\right)\eta(t,q)
+\partial_{2}\left(L[q]_{\tau}(t)-\lambda\cdot g[q]_{\tau}(t)\right)\cdot\xi(t,q)\\
&+\partial_{3}\left(L[q]_{\tau}(t)-\lambda\cdot
g[q]_{\tau}(t)\right)\cdot\left(\dot{\xi}(t,q)
-\dot{q}(t)\dot{\eta}(t,q)\right)+\left(L[q]_{\tau}(t)-\lambda\cdot g[q]_{\tau}(t)\right)\dot{\eta}(t,q)\Bigr]dt\\
&+\int_{t_1}^{t_2}\Bigl[\partial_{4}
\left(L[q]_{\tau}(t)-\lambda\cdot g[q]_{\tau}(t)\right)\cdot\xi(t-\tau,q(t-\tau))\\
&+\partial_{5}\left(L[q]_{\tau}(t)-\lambda\cdot
g[q]_{\tau}(t)\right)\cdot\left(\dot{\xi}(t-\tau,q(t-\tau))
-\dot{q}(t-\tau)\dot{\eta}(t-\tau,q(t-\tau))\right)\Bigr]dt= 0.
\end{split}
\end{equation}
Performing a linear change of variables $t=\sigma+\tau$ in the last
integral of \eqref{eq:cnsind3}, and keeping in mind that
$\xi=\eta=0$ on $[t_1-\tau,t_1]$, equation \eqref{eq:cnsind3}
becomes
\begin{equation}
\label{eq:cnsind}
\begin{split}
\int_{t_1}^{t_2-\tau}\Bigl[&-\dot{\Phi}[q]_{\tau}(t)+\partial_{1}
\left(L[q]_{\tau}(t)-\lambda\cdot g[q]_{\tau}(t)\right)\eta(t,q)\\
&+\left(\partial_{2} \left(L[q]_{\tau}(t)-\lambda\cdot
g[q]_{\tau}(t)\right)+\partial_4
\left(L[q]_{\tau}(t+\tau)-\lambda\cdot g[q]_{\tau}(t+\tau)\right)\right)\cdot\xi(t,q)\\
&+\left(\partial_{3}\left(L[q]_{\tau}(t)-\lambda\cdot
g[q]_{\tau}(t)\right)+\partial_5
\left(L[q]_{\tau}(t+\tau)-\lambda\cdot
g[q]_{\tau}(t+\tau)\right)\right)\cdot\left(\dot{\xi}(t,q)
-\dot{q}(t)\dot{\eta}(t,q)\right)\\
&+\left(L[q]_{\tau}(t)-\lambda\cdot g[q]_{\tau}(t)\right)\dot{\eta}(t,q)\Bigr]dt\\
&+ \int_{t_2-\tau}^{t_2}\Bigl[-\dot{\Phi}[q]_{\tau}(t)+\partial_{1}
\left(L[q]_{\tau}(t)-\lambda\cdot g[q]_{\tau}(t)\right)\eta(t,q)
+\partial_{2}
\left(L[q]_{\tau}(t)-\lambda\cdot g[q]_{\tau}(t)\right)\cdot\xi(t,q)\\
&+\partial_{3}\left(L[q]_{\tau}(t)-\lambda\cdot
g[q]_{\tau}(t)\right)\cdot\left(\dot{\xi}(t,q)
-\dot{q}(t)\dot{\eta}(t,q)\right)+\left(L[q]_{\tau}(t)-\lambda\cdot
g[q]_{\tau}(t)\right)\dot{\eta}(t,q)\Bigr]dt = 0.
\end{split}
\end{equation}
Taking into consideration that \eqref{eq:cnsind} holds for an arbitrary subinterval
$I \subseteq [t_1,t_2]$, equations \eqref{eq:cnsind1} and \eqref{eq:cnsind2} hold.
\end{proof}

\begin{theorem}[Noether's symmetry theorem with time delay for Lipschitz functions]
\label{theo:tnnd} If functional \eqref{agp} is invariant up to
$\Phi$ in the sense of Definition~\ref{def:invndLIP} such that
satisfy the condition \eqref{CDUR}, then the quantity
$C(t,t+\tau,q(t),q(t-\tau),q(t+\tau),\dot{q}(t),\dot{q}(t-\tau),\dot{q}(t+\tau))$
defined by
\begin{multline}
\label{eq:tnnd} -\Phi[q]_{\tau}(t)+\left(\partial_{3} F[q]_{\tau}(t)
+\partial_{5} F[q]_{\tau}(t+\tau)\right)\cdot\xi(t,q(t))\\
+\Bigl(F[q]_{\tau}-\dot{q}(t)\cdot(\partial_{3} F[q]_{\tau}(t)
+\partial_{5} F[q]_{\tau}(t+\tau))\Bigr)\eta(t,q(t))
\end{multline}
for $t_1\leq t\leq t_{2}-\tau$ and by
\begin{equation}
\label{eq:tnnd11} -\Phi[q]_{\tau}(t)+\partial_{3}
F[q]_{\tau}(t)\cdot\xi(t,q(t))
+\Bigl(F[q]_{\tau}-\dot{q}(t)\cdot\partial_{3}
F[q]_{\tau}(t)\Bigr)\eta(t,q(t))
\end{equation}
for $t_2-\tau< t\leq t_{2}\,,$ is a constant of motion with time
delay along any
%$q(\cdot)\in Lip\left([t_1-\tau,t_2];\mathbb{R}^n\right)$
$q(\cdot)\in Lip\left([t_1-\tau,t_2]\right)$ satisfying both
\eqref{eq:eldf11} and \eqref{eq:cdrnd}-\eqref{eq:cdrnd1}, i.e.,
along any Lipschitz Euler--Lagrange extremal that is also a
Lipschitz DuBois--Reymond extremal.
\end{theorem}

\begin{proof}
We prove the theorem in the interval $t_1\leq t\leq t_{2}-\tau$. The
proof is similar for the interval $t_2-\tau\leq t\leq t_{2}$.
Noether's constant of motion with time delay \eqref{eq:tnnd} follows
by using in the interval $t_1\leq t\leq t_{2}-\tau$ the
DuBois--Reymond condition with time delay \eqref{eq:cdrnd} and the
Euler--Lagrange equation with time delay \eqref{eq:eldf11} into the
necessary condition of invariance \eqref{eq:cnsind1}:
\begin{equation*}
\begin{split}
0&=\int_{t_1}^{t_2-\tau}\Bigl[-\dot{\Phi}[q]_{\tau}(t)+\partial_{1}
\left(L[q]_{\tau}(t)-\lambda\cdot
g[q]_{\tau}(t)\right)\eta(t,q)\\
&+\left(\partial_{2}
\left(L[q]_{\tau}(t)-\lambda\cdot g[q]_{\tau}(t)\right)+\partial_4 L[q]_{\tau}(t+\tau)\right)\cdot\xi(t,q)\\
& +\left(\partial_{3}\left(L[q]_{\tau}(t)-\lambda\cdot
g[q]_{\tau}(t)\right)+\partial_5
L[q]_{\tau}(t+\tau)\right)\cdot\left(\dot{\xi}(t,q)
-\dot{q}(t)\dot{\eta}(t,q)\right)\\
&+\left(L[q]_{\tau}(t)-\lambda\cdot g[q]_{\tau}(t)\right)\dot{\eta}(t,q)\Bigr]dt\\
&=
\int_{t_1}^{t_2-\tau}\Bigl[-\dot{\Phi}[q]_{\tau}(t)+\frac{d}{dt}\left(\partial_{3}
\left(L[q]_{\tau}(t)-\lambda\cdot g[q]_{\tau}(t)\right)
+\partial_5\left(L[q]_{\tau}(t+\tau)-\lambda\cdot
g[q]_{\tau}(t+\tau)\right)\right)\cdot\xi(t,q)\\
&\quad +\left(\partial_{3} \left(L[q]_{\tau}(t)-\lambda\cdot
g[q]_{\tau}(t)\right)+\partial_5\left(L[q]_{\tau}(t+\tau)-\lambda\cdot
g[q]_{\tau}(t+\tau)\right)\right)\cdot\dot{\xi}(t,q)\\
&\quad +\frac{d}{dt}\left\{L[q]_{\tau}(t)-\lambda\cdot
g[q]_{\tau}(t)-\dot{q}(t)\cdot(\partial_{3}
\left(L[q]_{\tau}(t)-\lambda\cdot g[q]_{\tau}(t)\right)\right.\\
&\left. +\partial_{5} \left(L[q]_{\tau}(t+\tau)-\lambda\cdot
g[q]_{\tau}(t+\tau)\right))\right\}\eta(t,q)\\
&\quad +\left\{L[q]_{\tau}(t)-\lambda\cdot
g[q]_{\tau}(t)-\dot{q}(t)\cdot(\partial_{3}
\left(L[q]_{\tau}(t)-\lambda\cdot
g[q]_{\tau}(t)\right)\right.\\
&\left. +\partial_{5} \left(L[q]_{\tau}(t+\tau)-\lambda\cdot
g[q]_{\tau}(t+\tau)\right))\right\}\dot{\eta}(t,q)\Bigr]dt,
\end{split}
\end{equation*}
that is,
\begin{multline}
\label{eq:cnsind11}
\int_{t_1}^{t_2-\tau}\frac{d}{dt}\Bigl[-\Phi[q]_{\tau}(t)
+\left(\partial_{3} L[q]_{\tau}(t)
+\partial_{5} L[q]_{\tau}(t+\tau)\right)\cdot\xi(t,q(t))\\
+\Bigl(L[q]_{\tau}(t)-\dot{q}(t)\cdot(\partial_{3} L[q]_{\tau}(t)
+\partial_{5} L[q]_{\tau}(t+\tau))\Bigr)\eta(t,q(t))\Bigr]dt = 0.
\end{multline}
Taking into consideration that \eqref{eq:cnsind11} holds for any
subinterval $I\subseteq [t_1,t_2-\tau]$, we conclude that
\begin{multline*}
-\Phi[q]_{\tau}(t)+\left(\partial_{3} L[q]_{\tau}(t)
+\partial_{5} L[q]_{\tau}(t+\tau)\right)\cdot\xi(t,q(t))\\
+\Bigl(L[q]_{\tau}(t)-\dot{q}(t)\cdot(\partial_{3} L[q]_{\tau}(t)
+\partial_{5} L[q]_{\tau}(t+\tau))\Bigr)\eta(t,q(t))=\text{constant}.
\end{multline*}
\end{proof}

\begin{example}

Consider the isoperimetric problem of the calculus of variations
with time delay
\begin{equation}
\label{eq:ex}
\begin{gathered}
J^{1}[q(\cdot)]=\int_0^3\left(\dot{q}(t)+\dot{q}(t-1)\right)^2dt \longrightarrow \min,\\
q(t)=-t \, ,~-1\leq t\leq 0, \quad q(3)=1,
\end{gathered}
\end{equation}
subject to  isoperimetric equality constraints
\begin{equation},\,m\geq1\,;
\label{CT5} I^{1}[q(\cdot)]=\int_{0}^{3}
\left(\dot{q}(t)\right)^2dt=l
\end{equation}
in the class of functions
%$q(\cdot)\in Lip\left([-1,3];\mathbb{R}\right)$
$q(\cdot)\in Lip\left([-1,3]\right)$.

For this example, the augmented
Lagrangian $F$ is given as
\begin{equation}\label{}
    F=\left(\dot{q}(t)+\dot{q}(t-1)\right)^2
    -\lambda\left(\dot{q}(t)\right)^2 \,.
\end{equation}
From Theorem~\ref{Thm:FractELeq1} (see Remark~\ref{re:EL}), one
obtains that any solution to problem \eqref{eq:ex}-\eqref{CT5} must
satisfy
\begin{equation}
\label{eq:ex:EL1}
2\dot{q}(t)+\dot{q}(t-1)+\dot{q}(t+1)
-2\lambda\left(\dot{q}(t)+\dot{q}(t+1)\right)=c_1,
\quad 0\leq t\leq 2,
\end{equation}
\begin{equation}
\label{eq:ex:EL2}
\dot{q}(t)+\dot{q}(t-1)-\lambda\dot{q}(t)=c_2,
\quad 2\leq t\leq 3,
\end{equation}
where $c_1$ and $c_2$ are constants.

Note that the functional integral $J^{1}$ and the  augmented
Lagrangian $F$ of our above problem satisfy all the hypotheses of the Theorems~\ref{thmtonelli} and ~\ref{th:RE} if $\lambda<1\,.$
Because problem
\eqref{eq:ex}--\eqref{CT5} is autonomous, we have invariance, in the
sense of Definition~\ref{def:invndLIP}, with $\eta\equiv 1$ and
$\xi\equiv 0$. Simple calculations show that isoperimetric Noether's
constant of motion with time delay
\eqref{eq:tnnd}--\eqref{eq:tnnd11} coincides with the
DuBois--Reymond condition \eqref{eq:cdrnd}--\eqref{eq:cdrnd1}:
\begin{multline}
\label{eq:ex:DBR1} \left(\dot{q}(t)+\dot{q}(t-\tau)\right)^2
    -\lambda\left(\dot{q}(t)\right)^2
    -2\dot{q}(t)\left[2\dot{q}(t)+\dot{q}(t-1)+\dot{q}(t+1)\right.\\
\left.-\lambda\left(\dot{q}(t)+\dot{q}(t+1)\right)\right]=c_3,\quad
0\leq t\leq 2,
\end{multline}
and 
\begin{equation}
\label{eq:ex:DBR2}\left(\dot{q}(t)+\dot{q}(t-\tau)\right)^2
    -\lambda\left(\dot{q}(t)\right)^2
    -2\dot{q}(t)\left[ \dot{q}(t)+\dot{q}(t-1)-\lambda\dot{q}(t)\right]
    =c_4,\quad 2\leq t\leq 3,
\end{equation}
where $c_3$ and $c_4$ are constants.

One can easily check that function
%$q(\cdot)\in Lip\left([-1,3];\mathbb{R}^n\right)$
$q(\cdot)\in Lip\left([-1,3]\right)$
defined by
\begin{equation}
\label{eq:ext:ex:22b} q(t)=
\begin{cases}
-t & ~\textnormal{for}~ -1< t\leq 0\\
t & ~\textnormal{for}~ 0< t\leq 1\\
-t+2 & ~\textnormal{for}~ 1< t\leq 2\\
t-2 & ~\textnormal{for}~ 2< t\leq 3
\end{cases}
\end{equation}
is an isoperimetric Euler--Lagrange extremal, i.e., satisfies
\eqref{eq:ex:EL1}--\eqref{eq:ex:EL2} and is also a isoperimetric
DuBois--Reymond extremal with $l=1$ and $\lambda=-1$ i.e., satisfies
\eqref{eq:ex:DBR1}--\eqref{eq:ex:DBR2}. Theorem~\ref{theo:tnnd}
asserts the validity of Noether's constant of motion, which is here
 verified: \eqref{eq:tnnd}--\eqref{eq:tnnd11} holds along
\eqref{eq:ext:ex:22b} with $\Phi\equiv 0$, $\eta\equiv 1$, and
$\xi\equiv 0$.

\end{example}

% ------------------------

% ------------------------

%%%%%%%%%%%%%%%%%%%%%%%%%%%%%%%%%%%%%%%%%%%%%%%%%%

\section{Conclusions and Open Questions}
\label{sec:Conc}

Isoperimetric problems are a classical and historical mathematical subject that is currently in strong research due to its numerous applications in physics and engineering~\cite{Bruce:book}. On the other hand, problems with time delays play a crucial role in the modeling of real-life phenomena in various fields of applications\cite{ChiLoi,EFridman,GoKeMa}. However, only recently the theory of conservation law in variational calculus with delay systems was initiated in \cite{GF2012} with the proof of delayed Noether's theorem. In the present work, we go a step further: we prove an isoperimetric Noether's theorem with time delay.

The isoperimetric variational theory with time delay is in its childhood so that much remains to be done. This is particularly true in the area of isoperimetric optimal control with time delay, where the results are rare. Here, an isoperimetric Lagrangian formulation with time delay is obtained. To the best of the author's knowledge, there is no general formulation of an isoperimetric version of Pontryagin's Maximum Principle. Then, with an isoperimetric notion of Pontryagin extremal, one can try to extend the present results to the more general context of isoperimetric optimal control with time delay.

\section*{Acknowledgements}

This work was supported in part by CNPq, CAPES and FAPERGS, brazilian funding agencies.

% ------------------------

% ------------------------

\end{document}